\documentclass{amsart}
\usepackage{amsthm}
\usepackage{amssymb}

\swapnumbers

\newtheorem{theorem}{Theorem}[section]
\newtheorem{proposition}[theorem]{Proposition}

\newtheorem{corollary}[theorem]{Corollary}
\newtheorem{lemma}[theorem]{Lemma}

\theoremstyle{definition}
\newtheorem*{definition}{Definition}

\theoremstyle{remark}
\newtheorem{remark}[theorem]{Remark}
\newtheorem{example}[theorem]{Example}
\newtheorem*{notation}{Notation}
\newtheorem*{notations}{Notations}

\newtheorem*{comment}{Comment}

\begin{document}
\title{Pseudo-diagonals and Uniqueness Theorems}
\author{Gabriel Nagy}
\address{Department of Mathematics, Kansas State university, Manhattan KS 66506, U.S.A.}
\email{nagy@math.ksu.edu}
\author{Sarah Reznikoff}
\address{Department of Mathematics, Kansas State university, Manhattan KS 66506, U.S.A.}
\email{sarahrez@math.ksu.edu}
\begin{abstract}
We examine a certain type of abelian C*-subalgebras that allow one to give a unified treatment of two uniqueness theorems: for graph C*-algebras and for certain reduced crossed products.
\end{abstract}
\maketitle

This note is meant to complement the paper \cite{NaRe}, which provides one of the two main examples of our results, by offering a conceptual treatment of the Uniqueness Theorem for the C*-algebras associated with graphs satisfying
condition (L), and its generalization found in \cite{Sz}.
As pointed out in many places in the graph C*-algebra literature  (see for instance \cite{Ren1}), for graphs satisfying condition (L), a natural abelian C*-subalgebra
(which is referred to in \cite{NaRe} as
the ``diagonal'') turns out to give substantial information about
the ambient (graph) C*-algebra.
A similar treatment was proposed by Kumijian in \cite{Kum}, where he introduced the notion of C*-diagonals.

In this paper we explain how, by considerably weakening several hypotheses
in Kumjian's definitions (in particular by getting rid of normalizers
entirely), one can still obtain several key results, the most
significant one being an ``abstract'' uniqueness property (see Theorem \ref{pdimpliesac} below).
Another illustration of this general approach is given in the context of
reduced crossed products of abelian C*-algebras by essentially free actions of discrete
groups, where we recover another uniqueness result due to Archbold and Spielberg \cite{AS}.

Our treatment focuses on the unique state extension property, as discussed
in \cite{KaSin} and \cite{And}, by weakening the global requirement
made in the so-called Extension Property discussed in \cite{And}.

\section{Notations and Preliminaries}

\begin{notations}
For any C*-algebra $\mathcal{A}$, we denote by $S(\mathcal{A})$ its set of states, and by
$P(\mathcal{A})$ the set of pure states.
Given a C*-subalgebra $\mathcal{B}\subset\mathcal{A}$, using the Hahn-Banach Theorem, it follows that any
$\phi\in S(\mathcal{B})$ has at least one extension to some $\psi\in S(\mathcal{A})$, so the set
$$S_\phi(\mathcal{A})=\{\psi\in S(\mathcal{A})\,:\,\psi\big|_{\mathcal{B}}=\phi\}$$
is non-empty, convex and weak*-compact.
Remark that, if $\phi\in P(\mathcal{B})$, then every extreme point in $S_\phi(\mathcal{A})$ is in fact an
extreme point in $S(\mathcal{A})$, thus the intersection
$S_\phi(\mathcal{A})\cap P(\mathcal{A})$ is non-empty.

With this observation in mind, we define the space
$$P_1(\mathcal{B}\!\uparrow\!\mathcal{A})=\{\phi\in P(\mathcal{B})\,\text{: card}\,S_\phi(\mathcal{A})=1\}.$$
\end{notations}

\begin{remark} \label{func}
It is fairly easy to see that the sets $P_1$ enjoy the following functorial property:
\begin{itemize}
\item[{\sc (f)}] {\em If $\pi:\mathcal{A}\to\mathcal{M}$ is a $*$-homomorphism, and
$\phi\in S(\pi(\mathcal{B}))$ is such that $\phi\circ\pi\in P_1(\mathcal{B}\!\uparrow\!\mathcal{A})$, for some C*-subalgebra $\mathcal{B}\subset\mathcal{A}$, then $\phi\in P_1(\pi(\mathcal{B})\uparrow\pi(\mathcal{A}))$.}
\end{itemize}
\end{remark}

Concerning the space $P_1(\mathcal{B}\!\uparrow\!\mathcal{A})$, one has the following important result of Anderson.

\begin{theorem}[{\cite[Thm.3.2]{And}}]
Assume $\mathcal{A}$ is unital and the C*-subalgebra $\mathcal{B}\subset\mathcal{A}$ contains the unit of $\mathcal{A}$.
For a pure state $\phi\in P(\mathcal{B})$, the following conditions are equivalent:
\begin{itemize}
\item[(i)] $\phi\in P_1(\mathcal{B}\!\uparrow\!\mathcal{A})$;
\item[(ii)] there exists a map
$\psi:\mathcal{A}\to\mathbb{C}$ (no extra condition assumed), such that
\begin{itemize}
\item[$(*)$] for every $a \in \mathcal{A}$, for every $\varepsilon>0$, there exists a positive element
$b\in\mathcal{B}$, such that $\|b\|=\phi(b)=1$, and
$\|bab-\psi(a)b^2\|<\varepsilon$.
\end{itemize}
\end{itemize}
Moreover, if $\psi$ satisfies condition $(*)$, then $\psi$ is the unique (pure) state on $\mathcal{A}$ that extends $\phi$.
\end{theorem}
\begin{comment}
Condition (ii) implies both (i) and the uniqueness of $\psi$, even in the non-unital case. Indeed, if $\psi$ satisfies condition $(*)$, then we can define, for each $a\in\mathcal{A}$ some sequence $\{b_n(a)\}_{n=1}^\infty\subset\mathcal{B}$ of positive elements with $\|b_n(a)\|=\phi(b_n(a))=1$, such that
$$\lim_{n\to\infty}\|b_n(a)ab_n(a)-\psi(a)b_n(a)^2\|=0,$$
so for every $\eta\in S_\phi(\mathcal{A})$ we will get
$$\lim_{n\to\infty}\big[\eta(b_n(a)ab_n(a))-\psi(a)\phi(b_n(a)^2)\big]=0.$$
Since for any positive element $b\in\mathcal{B}$, by the Cauchy-Schwartz Inequality we have $\phi(b)^2\leq\|\phi\|\cdot\phi(b^2)$, it follows that
$\phi(b_n(a)^2)=1$, so the above equality yields
\begin{equation}
\psi(a)=\lim_{n\to\infty}\eta(b_n(a)ab_n(a)).
\label{comment}
\end{equation}
Since for every $\eta\in S(\mathcal{A})$ and every positive element $x\in\mathcal{A}$
with $\|x\|=\eta(x)=1$, we have (again using Cauchy-Schwartz) $\eta((1-x)a)=\eta(a(1-x))=0$, thus $\eta(a)=\eta(xax)$, $\forall\,a\in\mathcal{A}$,
we see that $\eta(b_n(a)ab_n(a))=\eta(a)$, $\forall\,\eta\in S_\phi(\mathcal{A})$, so  the equality \eqref{comment} forces $S_\phi(\mathcal{A})=\{\psi\}$.
\end{comment}

An application of the above Theorem, which perhaps clarifies the status of the set $P_1(\mathcal{B}\!\uparrow\!\mathcal{A})$ a little better, in the case when
$\mathcal{B}$ is {\em abelian}, is as follows. (Although this result may be known, we were not able to find a proof in the literature, so one is supplied here, for the reader's convenience.)

\begin{proposition} \label{net}
Assume $\mathcal{B}$ is an abelian (possibly non-unital) C*-subalgebra of a
(possibly non-unital) C*-algebra $\mathcal{A}$. For a pure state $\phi\in P(\mathcal{B})$,
conditions {\rm (i)} and {\rm (ii)} from Theorem 1.1 are also equivalent to the condition
\begin{itemize}
\item[(ii')] there exists a map $\psi:\mathcal{A}\to\mathbb{C}$, and a net
$(b^{}_\lambda)_{\lambda\in\Lambda}$ of positive elements in $\mathcal{B}$, such that
$\|b^{}_\lambda\|=\phi(b^{}_\lambda)=1$, $\forall\,\lambda\in\Lambda$, and
$\lim_{\lambda\in\Lambda}\|b^{}_\lambda ab^{}_\lambda -\psi(a)b^2_\lambda\|=0$, $\forall\,a\in\mathcal{A}$.
\end{itemize}
As in Theorem 1.1, the map $\psi$ is the unique (pure) state extension of $\phi$.
\end{proposition}
\begin{proof}
Since the implication $(ii')\Rightarrow (ii)$ is trivial, while the implication
$(ii)\Rightarrow (i)$ has been discussed above, we only need to consider the implication
$(i)\Rightarrow (ii')$. Assume $\phi\in P(\mathcal{A})$ has a unique state
extension $\psi\in S(\mathcal{A})$.
Let us adjoin a (new) unit $\mathbf{1}$ to both $\mathcal{A}$ and $\mathcal{B}$, and let us consider the canonical extensions $\tilde{\phi}\in S(\tilde{\mathcal{B}})$ and $\tilde{\psi}\in
S(\tilde{\mathcal{A}})$ of $\phi$ and $\psi$ respectively, to the unitized algebras. Of course,
$\tilde{\phi}$ is a pure state on $\tilde{\mathcal{B}}$, and $\tilde{\psi}$ is its unique extension
to a (pure) state on $\tilde{\mathcal{A}}$, so we can now apply Theorem 1.1 to $\tilde{\phi}$ and $\tilde{\psi}$.
Choose then, for every $a\in\mathcal{A}$ and every $\varepsilon>0$, a positive element
$x^{}_{a,\varepsilon}\in \tilde{\mathcal{B}}$, so that
$\|x^{}_{a,\varepsilon}\|=\tilde{\phi}(x^{}_{a,\varepsilon})=1$, and such that
\begin{equation}
\|x^{}_{a,\varepsilon}ax^{}_{a,\varepsilon}-\psi(a)x_{a,\varepsilon}^2\|<\varepsilon.
\label{prop1-a}
\end{equation}
(Note that, since we only use $a\in\mathcal{A}$, $\tilde{\psi}$ is replaced by $\psi$.)

We define, for every finite set $\mathcal{F}=\{a_1,\dots,a_n\}\subset
\mathcal{A}$ and any $\varepsilon>0$, the element
$x^{}_{\mathcal{F},\varepsilon}=\prod_{j=1}^nx^{}_{a_j,\varepsilon}\in\tilde{\mathcal{B}}$. Since
$\tilde{\mathcal{B}}$ is abelian, this product is unambiguously defined and
positive. Furthermore, since all the
$x^{}_{a_j,\varepsilon}$'s have norm $1$, by suitably multiplying both sides of \eqref{prop1-a}, we also get
\begin{equation}
\|x^{}_{\mathcal{F},\varepsilon}ax^{}_{\mathcal{F},\varepsilon}-\psi(a)x_{\mathcal{F},\varepsilon}^2\|<\varepsilon,
\,\,\,\forall\,a\in\mathcal{F},
\label{prop1-b}
\end{equation}
with all $x^{}_{\mathcal{F},\varepsilon}$ having norm $\leq 1$.

Of course, since $\tilde{\phi}$ is a pure state on $\tilde{\mathcal{B}}$, which is abelian, it follows that
$\tilde{\phi}$ is multiplicative on $\tilde{\mathcal{B}}$, so in particular we also have
$\tilde{\phi}(x^{}_{\mathcal{F},\varepsilon})=1$, which in turn forces
$\|x^{}_{\mathcal{F},\varepsilon}\|=1$.

Up to this point, we almost have all that we want: a net $(x_\lambda)_{\lambda\in\Lambda}$ of positive
elements in $\tilde{\mathcal{B}}$, such that
\begin{itemize}
\item[(a)] $\|x^{}_\lambda\|=\tilde{\phi}(x^{}_\lambda)=1$, $\forall\,\lambda\in\Lambda$;
\item[(b)] $\lim_{\lambda\in\Lambda}\|x^{}_\lambda ax^{}_\lambda -\psi(a)x^2_\lambda\|=0$, $\forall\,a\in\mathcal{A}$.
\end{itemize}
To produce the desired $b_\lambda$'s (in $\mathcal{B}$), we regard $\mathcal{B}$ as $C_0(\Omega)$, for some
locally compact space $\Omega$, so $\phi$ will be an evaluation map $\phi:C_0(\Omega)\ni f\longmapsto f(\omega)\in\mathbb{C}$,
and then, if we take $f\in \mathcal{B}$ to be any positive element with
$\|f\|=1=f(\omega)$, then the assignment
$b_\lambda=fx_\lambda$ will do the job.
\end{proof}

\section{Pseudo-diagonals}

One key condition we wish to isolate is as follows.
\begin{definition}
Assume $\mathcal{A}$ is a C*-algebra. A C*-subalgebra $\mathcal{B}\subset\mathcal{A}$ is said to have the
{\em Almost Extension Property}, if the set $P_1(\mathcal{B}\!\uparrow\!\mathcal{A})$ is weak*-dense in $P(\mathcal{B})$.

Let us say that $\mathcal{B} \subseteq \mathcal{A}$ has the (honest) Extension Property if $P_1(\mathcal{B}\!\uparrow\!\mathcal{A}) =P(\mathcal{B})$; of course, if $\mathcal{B}$ has the Extension Property, then it also has the Almost Extension Property.
\end{definition}

\begin{example}
Given any set $\Sigma$, the C*-algebra $\ell^\infty(\Sigma)$, viewed as a C*-subalgebra of
$\mathcal{B}(\ell^2(\Sigma))$ (as multiplication operators), has the Almost Extension Property, since
all pure states of the form $ev_s:\ell^\infty(\Sigma)\ni f\longmapsto f(s)\in\mathbb{C}$, $s\in\Sigma$, have unique state extensions to $\mathcal{B}(\ell^2(\Sigma))$, and the set $\{ev_s\,:\,s\in\Sigma\}$ is dense in
$P(\ell^\infty(\Sigma))=\beta\Sigma$ (the Stone-Cech compactification of $\Sigma$). The famous Kadison-Singer Problem asks whether the inclusion $\ell^\infty(\mathbb{N})\subset\mathcal{B}(\ell^2(\mathbb{N}))$ has in fact the (honest) Extension Property.
\end{example}
\begin{remark}
Using the functorial property {\sc (f)} from Remark~\ref{func}, it is immediate that, whenever
 $\mathcal{B}\subset\mathcal{A}$ has the
honest Extension Property, and $\pi:\mathcal{A}\to\mathcal{M}$ is a $*$-homomorphism, it follows that $\pi(\mathcal{B})\subset\pi(\mathcal{A})$ also has the honest Extension Property.

This implication may fail, if we use the Almost Extension Property, as seen for instance in Example~\ref{Ctorus} below. However, again using Remark~\ref{func}, the implication does hold, if
$\pi$ is injective on $\mathcal{B}$.
\end{remark}

As Anderson has shown, in the case when $\mathcal{B}$ is an {\em abelian\/} C*-subalgebra of $\mathcal{A}$ with the
(honest) Extension Property, then there exists a conditional expectation of $\mathcal{A}$ onto $\mathcal{B}$, which is in fact unique.
This may no longer be the case if we consider only the Almost Extension Property, as seen in the following.

\begin{example} \label{Ctorus}
Consider $\mathcal{A}=C([0,1])$ (the algebra of continuous functions on $[0,1]$) and the algebra
$C(\mathbb{T})$ (the algebra of continuous functions on the unit circle), identified as
as C*-subalgebra of $\mathcal{A}$ as:
$$\mathcal{B}=\{f\in\mathcal{A}\,:\,f(0)=f(1)\}.$$
Obviously, all evaluation maps $ev_s:C(\mathbb{T})\ni f\longmapsto f(s)\in\mathbb{C}$, $s\in (0,1)$, have unique extensions to states on $C([0,1])$, but there is no conditional expectation of
$C([0,1])$ onto $\mathcal{B}$.

If we consider the $*$-homomorphism $\pi:\mathcal{A}\ni f\longmapsto (f(0),f(1))\in\mathbb{C}\oplus\mathbb{C}$, then obviously the inclusion
$\pi(\mathcal{B})\subset\pi(\mathcal{A})$ fails to have the Almost Extension Property.
\end{example}

\begin{remark}
If $\mathcal{B}\subset\mathcal{A}$ has the Almost Extension Property, and there exists a conditional expectation
of $\mathcal{A}$ onto $\mathcal{B}$, then it is unique. Indeed, if $E,F:\mathcal{A}\to\mathcal{B}$ are two conditional expectations, then for  $\phi\in P_1(\mathcal{B}\!\uparrow\!\mathcal{A})$, we have
$\phi\circ E=\phi\circ F$ (by uniqueness of state extensions), so for every $a\in\mathcal{A}$ we have
$\phi\big(E(a)-F(a)\big)=0$, $\forall\,\phi\in P_1(\mathcal{B}\!\uparrow\!\mathcal{A})$. By density this forces
$\phi\big(E(a)-F(a)\big)=0$, $\forall\,\phi\in P(\mathcal{B})$, and then we get
$$\phi\big(E(a)-F(a)\big)=0,\,\,\,\forall\,a\in\mathcal{A},\,\phi\in\mathcal{A}^*,$$
which forces $E=F$.
\end{remark}
\begin{definition} In the case where $\mathcal{B}\subset\mathcal{A}$ has the Almost Extension Property and there exists a conditional expectation of $\mathcal{A}$ onto $\mathcal{B}$, we declare the inclusion $\mathcal{B}\subset\mathcal{A}$ to have the {\em Canonical Almost Extension Property\/}, and the (unique) conditional expectation $E$ will be called the {\em associated expectation}.
\end{definition}

\begin{comment}
If $\mathcal{B}\subset\mathcal{A}$ has the Almost Extension Property, the ``candidate'' for the associated
conditional expectation $E$ is constructed as follows. Start off by considering the unique state extension as
a map
$\theta_0:P_1(\mathcal{B}\!\uparrow\!\mathcal{A})\to\mathcal{A}^*$, and define its natural linear continuous extension
$\theta:\ell^1(P_1(\mathcal{B}\!\uparrow\!\mathcal{A}))\to\mathcal{A}^*$. Then we take the dual map
\begin{equation}
\theta^*:\mathcal{A}^{**}\to\ell^\infty(P_1(\mathcal{B}\!\uparrow\!\mathcal{A})),
\label{theta}
\end{equation}
together with the
embedding
\begin{equation}
J:\mathcal{B}\ni b\longmapsto(\phi(b))_{\phi\in P_1(\mathcal{B}\!\uparrow\!\mathcal{A})}\in \ell^\infty(P_1(\mathcal{B}\!\uparrow\!\mathcal{A})).
\label{J}
\end{equation}
If $\theta^*(\mathcal{A})\subset\text{Range}\,J$, then $E$ is simply $J^{-1}\circ \theta^*\big|_{\mathcal{A}}$.
Note that both maps \eqref{theta} and \eqref{J} are completely positive (because they are positive and take values in a commutative C*-algebra), and thus they are both contractive.
\end{comment}

\begin{remark}
In the case when $\mathcal{B}$ is {\em abelian}, the map $J$ defined by \eqref{J} is an isometric $*$-homo\-mor\-phism, so by Tomiyama's Theorem \cite{To}, the condition $\theta^*(\mathcal{A})\subset\text{Range}\,J$ is sufficient to conclude that $E$ is a conditional expectation.
\end{remark}

For future use, let us record the following technical result.
\begin{lemma} \label{seminorm}
Assume $\mathcal{B}$ is abelian and we have an inclusion $\mathcal{B}\subset\mathcal{A}$ with the Canonical Almost Extension Property, with associated expectation $E$. If we consider the seminorm $p:\mathcal{A}\to [0,\infty)$,
defined by
$$p(a)=\sup\left\{\|ab-ba\|\,:\,b\in\mathcal{B},\,\|b\|\leq 1\right\},\,\,\,a\in\mathcal{A},$$
then
\begin{equation}
\|E(a)^2-E(a^2)\|\leq 2\|a\|\cdot p(a),\,\,\,\forall\,a\in\mathcal{A}.
\label{lemma1}
\end{equation}
\end{lemma}
\begin{proof}
Fix $a\in \mathcal{A}$, as well as some $\phi\in P_1(\mathcal{B}\uparrow\mathcal{A})$, so that the unique state on $\mathcal{A}$ that extends $\phi$ is $\phi\circ E$.
According to Proposition 1.2,
there exists a net $(b^{}_\lambda)_{\lambda\in\Lambda}$ of positive elements in $\mathcal{B}$, with $\|b^{}_\lambda\|=\phi(b^{}_\lambda)=1$, so that
\begin{align}
\lim_{\lambda\in\Lambda}\|b^{}_\lambda a b^{}_\lambda-\phi(E(a))b^2_\lambda\|&=0;\label{lema1-1}\\
\lim_{\lambda\in\Lambda}\|b^{}_\lambda a^2 b^{}_\lambda-\phi(E(a^2))b^2_\lambda\|&= 0.\label{lema1-2}
\end{align}
One the one hand, since we have
\begin{align*}
&\|[b_\lambda^{}ab_\lambda^{}]^2-[\phi(E(a))b_\lambda^2]^2\|=\\
&\qquad=\|b^{}_\lambda ab_\lambda^{}[
b^{}_\lambda ab_\lambda^{}-\phi(E(a))b_\lambda^2]+[
b^{}_\lambda ab_\lambda^{}-\phi(E(a))b_\lambda^2]\phi(E(a))b_\lambda^2\|\leq\\
&\qquad \leq \left[\|b^{}_\lambda ab_\lambda^{}\|+\|\phi(E(a))b_\lambda^2\|\right]\cdot\|b^{}_\lambda ab_\lambda^{}-\phi(E(a))b_\lambda^2\|\leq \\
&\qquad\leq 2\|a\|\cdot\|b^{}_\lambda ab_\lambda^{}-\phi(E(a))b_\lambda^2\|,\end{align*}
we also get
\begin{equation}
\lim_{\lambda\in\Lambda}\|(b_\lambda^{}ab_\lambda^{})^2-[\phi(E(a))b_\lambda^2]^2\|=0.
\label{lema1-3}
\end{equation}
On the other hand, if we multiply (inside the norm) the left-hand side of
\eqref{lema1-2} on each side by $b^{}_\lambda$, we also get
\begin{equation}
\lim_{\lambda\in\Lambda}\|b^2_\lambda a^2 b^2_\lambda-\phi(E(a^2))b^4_\lambda\|= 0.\label{lema1-4}
\end{equation}
Finally, since we have
\begin{align*}
&\|b^2_\lambda a^2 b^2_\lambda-[b_\lambda^{}ab_\lambda^{}]^2\|=
\|b_\lambda^2 a[ab^{}_\lambda-b^{}_\lambda a]b^{}_\lambda-
b^{}_\lambda[ab^{}_\lambda-b^{}_\lambda a]b^{}_\lambda ab^{}_\lambda\|\leq\\
&\qquad\leq\left[\|b_\lambda^2 a\|\cdot\|b^{}_\lambda\|+\|b^{}_\lambda\|\cdot\|b^{}_\lambda a b^{}_\lambda\|\right]\cdot\|ab^{}_\lambda-b^{}_\lambda a\|\leq
2\|a\|\cdot p(a),
\end{align*}
we now get
\begin{align*}
&\|[\phi(E(a))b^2_\lambda]^2-\phi(E(a^2))b^4_\lambda\|\leq
\|b^2_\lambda a^2 b^2_\lambda-\phi(E(a^2))b^4_\lambda\|+\\
&\qquad\qquad+\|[b^{}_\lambda ab^{}_\lambda]^2-[\phi(E(a))b^2]^2\|+\|b^2_\lambda a^2 b^2_\lambda-
[b^{}_\lambda a b{}_\lambda]^2\|\leq\\
&\qquad\leq
\|b^2_\lambda a^2 b^2_\lambda-\phi(E(a^2))b^4_\lambda\|+\|(b^{}_\lambda ab^{}_\lambda)^2-[\phi(E(a))b^2]^2\|+2\|a\|\cdot p(a),
\end{align*}
so when we use
 \eqref{lema1-3} and \eqref{lema1-4}, we
get
\begin{equation}
\limsup_{\lambda\in\Lambda}\|[\phi(E(a))b^2_\lambda]^2-\phi(E(a^2))b^4_\lambda\|\leq
2\|a\|\cdot p(a).
\label{ineq1}
\end{equation}
Of course, when we apply $\phi$ to the quantity inside the norm, and using the obvious equality $\phi(b_\lambda^4)=\phi(b_\lambda^{})^4=1$, we get
$$\phi([\phi(E(a))b^2_\lambda]^2-\phi(E(a^2))^2b^4_\lambda)=\phi(E(a^2)-E(a)^2),$$
so the inequality \eqref{ineq1} yields
$$|\phi(E(a^2)-E(a)^2)|\leq 2\|a\|\cdot p(a).$$
If we take supremum over all $\phi\in P_1(\mathcal{B}\uparrow \mathcal{A})$, by density, we get \eqref{lemma1}.
\end{proof}

\begin{definition}
Suppose $\mathcal{A}$ is a C*-algebra. An abelian C*-subalgebra $\mathcal{B}\subset\mathcal{A}$ is called a {\em pseudo-diagonal in $\mathcal{A}$}, if
\begin{itemize}
\item[(i)] the inclusion $\mathcal{B}\subset\mathcal{A}$ has the Canonical Almost Extension Property;
\item[(ii)] the associated expectation $E:\mathcal{A}\to\mathcal{B}$ is {\em faithful}: i.e., it satisfies the implication
$E(a^*a)=0\Rightarrow a=0$.
\end{itemize}
\end{definition}

Condition (i) alone is not sufficient, as seen in the following

\begin{example} \label{Enotfaithfulexample}
Consider $\mathcal{B}=C([0,1])$, the characteristic function $f_0=\chi_{\{0\}}\in\ell^\infty([0,1])$, and the C*-subalgebra $\mathcal{A}=\mathcal{B}+\mathbb{C}f_0$, which can be characterized as the algebra of all bounded functions on $[0,1]$ that are continuous on $(0,1]$ and have a limit
at $0$. As it turns out, the inclusion $\mathcal{B}\subset\mathcal{A}$ has the Canonical Almost Extension Property. Indeed, all maps $ev_s\in P(\mathcal{B})$, $s\in (0,1]$ have unique state extensions, and we do have a conditional expectation $E:\mathcal{A}\to\mathcal{B}$, defined as
$$(E(f))(t)=\left\{\begin{array}{l}f(t)\text{, if }t>0\\
\lim_{s\to 0}f(s)\text{, if }t=0\end{array}\right.$$
However, $E$ fails to be faithful, since $E(f_0)=0$.
\end{example}

\begin{comment}
The use of the term ``pseudo-diagonal'' in our definition is meant to hint towards Kumjian's definition of so-called {\em C*-diagonals}, which includes the existence of a faithful conditional expectation of $\mathcal{A}$ onto $\mathcal{B}$, plus a certain feature of the normalizer set
\begin{equation}
\mathcal{N}(\mathcal{B})=\{a\in\mathcal{A}\,:\,a\mathcal{B}a^*\cup a^*\mathcal{B}a\subset\mathcal{B}\}.
\label{norm}
\end{equation}
(Actually, Kumjian uses a certain distinguished subset in
$\mathcal{N}(\mathcal{B})$, which he requires to generate $\mathcal{A}$ as a C*-algebra.) As shown in \cite{Kum}, C*-diagonals have in fact the (honest) Extension Property, thus C*-diagonals are in fact pseudo-diagonals.
\end{comment}

\begin{example} \label{coreex}
Assume $\mathcal{A}=C^*(E)$, the graph C*-algebra of some countable
graph $E=(E^0,E^1)$, which is the universal C*-algebra generated by
$\{P_v\}_{v\in E^0}\cup\{S^{}_e\}_{e\in E^1}$, subject to the Cuntz-Krieger relations
\begin{align}
&P^*_v=P^{}_v=P^2_v,\,\,\, \forall\,v\in \mathcal{E}^0;\label{CK1}\\
& S_e^*S^{}_e=P_{s(e)},\,\,\, \forall\,e\in \mathcal{E}^1;\label{CK2}\\
& S_e^*S^{}_f=0\text{, for any two {\em distinct\/} edges }e,f\in\mathcal{E}^1\label{CK3}\\
&\text{if $v\in \mathcal{E}^0$ has $r^{-1}(v)$ {\em non-empty and finite}, then }S^{}_v=\sum_{e\in r_1^{-1}(v)}S^{}_eS_e^*.\label{CK4}
\end{align}
(Here $s,r:E^1\to E^0$ denote the source and range maps.)
Define for every path $\alpha=e_1e_2\dots e_n$ (the
convention here is as in \cite{Ra}, i.e., $s(e_i)=r(e_{i+1})$,
$\forall\,i$) the projection $R_\alpha=S^{}_{e_1}\cdots
S^{}_{e_n}S^*_{e_n}\cdots S^*_{e_1}$ (with the convention that if
$\alpha=v\in E^0$, i.e., $\alpha$ is a path of length zero, then
$R_\alpha=P_v$), and define the ``diagonal'' algebra $\Delta(E)$ to be the
(abelian) C*-subalgebra generated by all $R_\alpha$'s.

As it turns out, $\Delta(E)$ is a pseudo-diagonal if and only if the
following condition is satisfied:
\begin{itemize}
\item[(L)] every cycle in $E$ has an entry.
\end{itemize}
As pointed out in \cite{NaRe}, in the absence of condition (L),
the correct substitute for $\Delta(E)$ is its commutant
$$M(E)=\{x\in C^*(E)\,:\,xa=ax,\,\,\forall\,a\in \Delta(E)\},$$
which turns out to be a pseudo-diagonal.
\end{example}
\begin{comment}
It should be pointed out here that, even in very simple cases, such as the one described
below, the pseudo-diagonal $M(E)$ may fail to have the (honest) Extension Property, so
there are instances when $P_1(\mathcal{B}\!\uparrow\!\mathcal{A})$, although
dense in $P(\mathcal{B})$, fails to coincide with $P(\mathcal{B})$.
In other words, there are pseudo-diagonals which are not C*-diagonals.

For instance, if one considers the Toeplitz C*-algebra $\mathcal{T}$
(which can be identified as $C^*(E)$ for a graph with two vertices $v$ and $w$, and two
edges $e$ and $f$ with $s(e)=v$ and $r(e)=s(f)=r(f)=w$)
the pseudo-diagonal $\mathcal{B}=\Delta(E)=M(E)$ is identified with
$C(\mathbb{N}\cup\{\infty\})$, where
$\mathbb{N}\cup\{\infty\}$ is the Alexandrov compactification of $\mathbb{N}$,
and $P_1(\mathcal{B}\!\uparrow\!\mathcal{T})=\mathbb{N}\subsetneq
\mathbb{N}\cup\{\infty\}=P(\mathcal{B})$.
 \end{comment}


In preparation for our next example (see Theorem~\ref{pdimpliesac} below), let us introduce the following.

\begin{notations}
Assume $X$ is some locally compact (Hausdorff) space, and $G$ is a (discrete) group
acting on $X$ (by homeomorphisms):
\begin{equation}
G\times X\ni (g,x)\longmapsto g\cdot x\in X
\label{G-X}
\end{equation}
Associated with this action, one defines the full crossed product
$C_0(X)\times G$ as follows.
Start off with the convolution $*$-algebra $C_0(X)[G]$ of finite formal sums of the form
$A=\sum_{g\in G}\mathbf{u}_ga_g$, with $a_g\in C_0(X)$, and $\{\mathbf{u}_g\}_{g\in G}$, some formal unitaries, satisfying the identity
\begin{equation}
\mathbf{u}_ga\mathbf{u}_{g^{-1}}=\alpha_g(a),\,\,\,a\in C_0(X),\,g\in G,
\label{aualpha}
\end{equation}
where $\alpha_g:C_0(X)\to C_0(X)$, $g\in G$, are the automorphisms given by
$$(\alpha_g(a))(x)=a(g^{-1}\cdot x), \,\,\,a\in C_0(X),\,x\in X,\,g\in G.$$
One equips this $*$-algebra with the maximal C*-norm $\|\,\cdot\,\|_{\max}$ (which can be shown to exist), and defines the C*-algebra $C_0(X)\times G$ to be the completion of $C_0(X)[G]$ in this norm. Of course, $C_0(X)\times G$ contains $C_0(X)$ as a C*-subalgebra, and
it is not hard to see that there exists a conditional expectation
$E$ of $C_0(X)\times G$ onto $C_0(X)$, hereafter referred to as the {\em standard\/}
expectation,
which acts on the dense subalgebra
$C_0(X)[G]$ as
\begin{equation}
E(\sum_{g\in G}\mathbf{u}_ga_g)=a_e.
\label{expG}
\end{equation}
In general, $E$ may fail to be faithful; one way to correct this is to consider the  {\em reduced\/} norm C*-norm $\|\,\cdot\,\|_{\text{red}}$ on $C_0(X)[G]$,   so that, when we take $C_0(X)\times_{\text{red}}G$ to be the completion, we have a faithful conditional
expectation $E_{\text{red}}$ acting as in \eqref{expG}.

By construction, we have a surjective $*$-homomorphism
$\pi_{\text{red}}:C_0(X)\times G\to C_0(X)\times_{\text{red}} G$, satisfying the identity
$E_{\text{red}}\circ\pi_{\text{red}}=E$.

Since $\pi_{\text{red}}$ acts as the identity on $C_0(X)$, by Remark 1.1, we have the inclusion
\begin{equation}
P_1(C_0(X)\uparrow C_0(X)\times G)\subset
P_1(C_0(X)\uparrow C_0(X)\times_{\text{red}} G).
\label{P1incG}
\end{equation}
%
%
%
%
\end{notations}

\begin{lemma} \label{extendev}
Use the notations as above, and fix some point $x\in X$.
Let $\psi$ be a state on $C_0(X)\times G$ that extends the evaluation state
$\text{\rm ev}_x:C_0(X)\ni a\longmapsto a(x)\in\mathbb{C}$.
If $g\in G$ is such that $g\cdot x\neq x$, then
$$\psi(\mathbf{u}_ga)=0,\,\,\,\forall\,a\in C_0(X).$$
\end{lemma}
\begin{proof}
Fix $a\in C_0(X)$ and $g\in G$ as above, so we also have $g^{-1}\cdot x\neq x$.
Let $f\in C_0(X)$ be a non-negative function, with $\|f\|=1$, $f(x)=1$ and $f(g^{-1}\cdot x)=0$.

On the one hand, using the Cauchy-Schwarz Inequality -- in the form
$|\psi(T)|^2\leq\psi(T^*T)$ -- we have:$
|\psi(\mathbf{u}_g(a-af)|^2
\leq\psi\left(|a-af|^2\right)= |a- af|^2(x)=0$,
so we get
\begin{equation}
\psi(\mathbf{u}_ga)=\psi(\mathbf{u}_gaf).
\label{lema-g1}
\end{equation}

On the other hand, using the Cauchy-Schwarz Inequality -- in the form $|\psi(T)|^2\leq\psi(TT^*)$ -- we also have
$$|\psi(\mathbf{u}_ga)|^2=|\psi(\mathbf{u}_gaf)|^2\leq\psi\left(\mathbf{u}_g|af|^2\mathbf{u}^*_g\right),$$
so using \eqref{aualpha} we obtain
$$|\psi(\mathbf{u}_ga)|^2\leq\psi\left(\alpha_{g}(|af|^2)\right)=
\text{ev}_x\left(\alpha_{g}(|af|^2)\right)=
|af|^2(g^{-1}\cdot x)=0,$$
with the last equality due to the fact that $f(g^{-1} \cdot x)=0$.
\end{proof}

\begin{notation}
For every $g\in G\smallsetminus\{e\}$, let us consider the set moved by $g$:
$$D_g=\{x\in X\,:\,g\cdot x\neq x\}.$$
Of course, all sets
$D_g$, $g\in G\smallsetminus\{e\}$ are open in $X$.
\end{notation}

\begin{corollary} \label{evextends}
With the notations as above, for every $x\in \bigcap_{g\in G\smallsetminus\{e\}}D_g$, the evaluation state $\text{\rm ev}_x\in P(C_0(X))$ belongs to $P_1\big(C_0(X)\uparrow C_0(X)\times G\big)$.
\end{corollary}
\begin{proof}
Fix $x\in \bigcap_{g\in G\smallsetminus\{e\}}X_g$, as well as some state
$\psi$ on $C_0(X)\times G$ that extends $\text{ev}_x$, and let us prove the
equality
\begin{equation}
\psi=\text{ev}_x\circ E.
\label{lemaG}
\end{equation}
This follows immediately from Lemma~\ref{extendev}, which implies that for any finite sum $A=\sum_g\mathbf{u}_ga_g\in C_0(X)[G]$, we have $\psi(A)=\psi(a_e)=\text{ev}_x(a_e)=\text{ev}_x(E(A))$, so the states $\psi$ and $\text{ev}_x\circ E$ agree on a dense C*-subalgebra.
\end{proof}

\begin{definition}
An action \eqref{G-X} is said to be {\em essentially free},
if for every $g\in G\smallsetminus\{e\}$, the set
$X_g=\{x\in X\,:\,g\cdot x=x\}$ has empty interior.
Equivalently, for every $g\in G\smallsetminus\{e\}$, the set $D_g$ is dense in $X$.
\end{definition}

\begin{theorem} \label{crossprod}
If the action of a countable group $G$   on $X$ is essentially free, then
\begin{itemize}
\item[(i)] the inclusion $C_0(X)\subset C_0(X)\times G$ has the Canonical Almost Extension
Property, with $E$ as the associated expectation;
\item[(ii)] $C_0(X)$ is a pseudo-diagonal in $C_0(X)\times_{\text{\rm red}}G$, with associated expectation $E_{\text{\rm red}}$.
\end{itemize}
\end{theorem}
\begin{proof}
Using \eqref{P1incG}, it suffices to prove only condition (i). By Corollary~\ref{evextends}, it
suffices to prove that the set $\bigcap_{g\in G\smallsetminus\{e\}}D_g$ is dense in $X$, or equivalently, the union of the complements $\bigcup_{g\in G\smallsetminus\{e\}}X_g$ has empty interior. Since each $X_g$ is closed with empty interior, and the union is countable, the desired conclusion follows from Baire's Theorem for locally compact spaces.
\end{proof}

%
%

\begin{remark}
If a countable group $G$ acts essentially freely on $X$, then Remark~\ref{func} (the last statement) yields the following fact:
\begin{itemize}
\item[$(*)$] {\em Whenever $(\pi,u)$ is a covariant representation of $(C_0(X),G)$ with $\pi$ injective, it follows that the inclusion
$\pi(C_0(X))\subset C^*(\pi(C_0(X))\cup u(G))$ has the Almost Extension Property}.
\end{itemize}
\end{remark}

We conclude this section with an off-topic discussion of a variant of the results from 2.8 through 2.10, which answers a question of Paulsen \cite{Pa}. The starting point is the following property, which is much stronger than essential freeness.
\begin{definition}
One says that an action \eqref{G-X} is {\em free}, if $D_g=X$, $\forall\,g\in G\smallsetminus\{e\}$.
\end{definition}
\begin{remark}
If $G$ acts freely on $X$, even when $G$ is uncountable, by the same proof as in Corollary~\ref{evextends} it follows that the inclusion
$C_0(X)\subset C_0(X)\times G$ has the honest Extension Property. Furthermore, as argued Remark~\ref{func}, this gives the following fact:
\begin{itemize}
\item[$(**)$] {\em Whenever $(\pi,u)$ is a covariant represen\-tation of $(C_0(X),G)$ (no extra assumption), it follows that the inclusion
$\pi(C_0(X))\subset C^*(\pi(C_0(X))\cup u(G))$ has the honest Extension Property}.
\end{itemize}
Of course, when we consider the reduced crossed product, i.e.~the quotient $*$-homomorphism $\pi_{\text{red}}$ it follows that
$C_0(X)$ is a pseudo-diagonal in $C_0(X)\times_{\text{red}}G$, with the honest Extension Property.
\end{remark}

\begin{example}
Fix a discrete group $G$ and identify $\ell^\infty(G)=C(\beta G)$.
We claim that, when we let $G$ act on $\ell^\infty(G)$ by translation, the corresponding action of $G$ on $\beta G$ is free. To see this, we fix some $g\in G\smallsetminus\{e\}$, as well as some point $x\in\beta G$ represented by an ultrafilter $\mathcal{U}$ in $G$, and we show that the ultrafilter that correponds to $g\cdot x$, which is obviously $g\mathcal{U}=\{gU\,:\,U\in\mathcal{U}\}$ is distinct from $\mathcal{U}$.
Let $\Gamma$ be the subgroup generated by $g$, and let
$T\subset G$ be a complete set of representatives for the quotient space
$\Gamma/ G$, so that every element $h\in G$ can be written uniquely as
$h=st$, with $s\in \Gamma$ and $t\in T$. Consider now the following two cases:
(i) $\Gamma$ is finite, say $\Gamma=\{e,g,g^2,\dots,g^k\}$; (ii) $\Gamma$ is cyclic,
so for every element $h\in G$, there exists a unique integer $n(h)\in\mathbb{Z}$ and a unique $t\in T$, such that $h=g^{n(h)}t$.

Case (i) is trivial, because we have a finite partition $G=T\cup gT\cup\dots g^k T$, so there exists some $j\in\{0,\dots,k\}$ such that the set $U=g^jT\in\mathcal{U}$, in which case the set $gU=g^{j+1}T$ (which belongs to $g\mathcal{U}$), cannot belong
to $\mathcal{U}$, simply because $U\cap gU=\varnothing$.

In case (ii) we use the map $n:G\to \mathbb{Z}$ to partitition
$G=G_{e}\cup G_{o}$, where
\begin{align*}
G_{e}&=\{h\in G\,:\,n(h)\text{ even}\},\\
G_{o}&=\{h\in G\,:\,n(h)\text{ odd}\},
\end{align*}
which satisfy the equalities $gG_e=G_o$ and $gG_o=G_e$, so either $G_e$ belongs to $\mathcal{U}$, in which case $G_o=gG_o$ belongs to $g\mathcal{U}$, but not to $\mathcal{U}$, or vice-versa.

To summarize, we have shown that then
inclusion $\ell^\infty(G)\subset\ell^\infty(G)\times G$ has the honest Extension Property, and furthermore, for any covariant representation $(\pi,u)$ of
$(\ell^\infty(G),G)$, the inclusion $\pi(\ell^\infty(G))\subset C^*(\pi(\ell^\infty(G))\cup u(G))$ also has the honest Extension Property.
\end{example}

\section{Abelian Cores}

In this section we derive the main results from this article, which deal with the automatic maximality of a pseudo-diagonal as an abelian subalgebra as well as a uniqueness property for the ambient C*-algebra.
As pointed out, for instance in \cite{KaSin}, if $\mathcal{B}\subset\mathcal{A}$ has
the (honest) Extension Property (with $\mathcal{B}$ abelian, of course), then $\mathcal{B}$ is a maximal abelian
subalgebra.
As we saw in Example~\ref{Enotfaithfulexample}, this may not necessarily be the case if only
the Canonical Almost Extension Property is assumed. As we shall prove shortly,
pseudo-diagonals are in fact maximal abelian. In addition to this feature, pseudo-diagonals also have another key property, geared toward uniqueness for the ambient C*-algebra $\mathcal{A}$, so it is convenient to introduce the following terminology.

\begin{definition}
Assume $\mathcal{A}$ is a C*-algebra. An abelian C*-subalgebra $\mathcal{B}\subset\mathcal{A}$ is called an {\em abelian core for $\mathcal{A}$}, if the following conditions are satisfied
\begin{itemize}
\item[{\sc (e${}_1$)}] there exists a unique conditional expectation $E$ of $\mathcal{A}$ onto $\mathcal{B}$;
\item[{\sc (e${}_2$)}] the conditional expectation $E$ of {\sc (e${}_1$)} is faithful;
\item[{\sc (m)}] $\mathcal{B}$ is a maximal abelian subalgebra in $\mathcal{A}$;
\item[{\sc (u)}] a $*$-representation $\pi:\mathcal{A}\to B(H)$ is injective whenever its restriction to $\mathcal{B}$ is injective.
\end{itemize}
\end{definition}

\begin{comment}
Although we are mainly interested in the implication
``pseudo-diagonal'' $\Rightarrow$ ``abelian core,''
it will be beneficial to
prove a slight generalization (Theorem \ref{pdimpliesac} below), in order to accomodate the following situation. Assume a C*-algebra $\mathcal{A}$ can be written as
\begin{equation}
\mathcal{A}=\overline{\bigcup_{\sigma\in\Sigma}\mathcal{A}_\sigma},
\label{A=union}
\end{equation}
where
$(\mathcal{A}_\sigma)_{\sigma\in\Sigma}$ is an increasing net of C*-subalgebras
(i.e., $\nu\succ\sigma\Rightarrow\mathcal{A}_\sigma\subset\mathcal{A}_\nu$). Assume also that we have an abelian C*-subalgebra $\mathcal{B}\subset\mathcal{A}$, which is contained in all $\mathcal{A}_\sigma$'s. It is pretty clear that the net
$(P_1(\mathcal{B}\uparrow\mathcal{A}_\sigma))_{\sigma\in\Sigma}$ is decreasing, and furthermore, we have the equality
$$\mathcal{P}_1(\mathcal{B}\uparrow\mathcal{A})=\bigcap_{\sigma\in\Sigma}P_1(\mathcal{B}\uparrow\mathcal{A}_\sigma),$$
so, even in the case when all inclusions $\mathcal{B}\subset\mathcal{A}_\sigma$,
$\sigma\in\Sigma$, have the Almost Extension Property, the inclusion $\mathcal{B}\subset\mathcal{A}$ might fail to have it.
\end{comment}

\begin{theorem} \label{pdimpliesac}
Assume $\mathcal{A}$ is a C*-algebra, written as in \eqref{A=union}, for an increasing net $(\mathcal{A}_\sigma)_{\sigma\in \Sigma}$ of C*-algebras. Assume also that
$\mathcal{B}\neq\{0\}$ is an abelian C*-subalgebra in $\bigcap_{\sigma\in\Sigma}\mathcal{A}_\sigma$, such that there exists
a faithful conditional expectation $E$ of $\mathcal{A}$ onto $\mathcal{B}$.
If all inclusions $\mathcal{B}\subset\mathcal{A}_\sigma$, $\sigma\in\Sigma$, have the Almost Extension Property, then $\mathcal{B}$ is an abelian core in $\mathcal{A}$.
\end{theorem}
\begin{proof}
Denote, for each $\sigma\in\Sigma$, the restriction of $E$ to $\mathcal{A}_\sigma$ by
$E_\sigma$. It is   clear that $\mathcal{B}$ is a pseudo-diagonal in
$\mathcal{A}_\sigma$ (with associated expectation $E_\sigma$), for each $\sigma\in\Sigma$.

The proof will be carried on in steps that correspond to the conditions listed in the preceding definition.

{\sc (e${}_1$)-(e${}_2$)} Assume $F$ is another conditional expectation of $\mathcal{A}$ onto $\mathcal{B}$. It is clear that, for each $\sigma\in\Sigma$, the restriction $F_\sigma=
F\big|_{\mathcal{A}_\sigma}:\mathcal{A}_\sigma\to\mathcal{B}$ is a conditional expectation of
$\mathcal{A}_\sigma$ onto $\mathcal{B}$, so using Remark 2.1 (for the inclusion $\mathcal{B}\subset\mathcal{A}_\sigma$), it follows that $F_\sigma=E_\sigma$. This argument shows that $F$ coincides with $E$ on the dense subalgebra $\bigcup_{\sigma\in\Sigma}\mathcal{A}_\sigma$, so by continuity
we have $F=E$, so both conditions {\sc (e${}_1$) (e${}_2$)} are satisfied.

{\sc (m)} All we have to prove here is that the commutant $$\mathcal{B}'=\{a\in\mathcal{A}\,:\,ab=ba,\,\,\forall\,b\in\mathcal{B}\}$$ is contained in $\mathcal{B}$. Since $\mathcal{B}'$ is a C*-algebra, it suffices to show that all self-adjoint elements $x\in\mathcal{B}'$ belong to $\mathcal{B}$. Fix such an element $x$, and
choose for every integer $n\geq 1$ an element
$a_n\in\bigcup_{\sigma\in\Sigma}\mathcal{A}_\sigma$, with
 $\|x-a_n\|\leq \frac{1}n$.
Using the seminorm $p$ from Lemma~\ref{seminorm}, we have $\big|p(x)-p(a_n)\big|\leq 2\|x-a_n\|\leq \frac{2}n$, which implies $p(a_n)\leq \frac{2}n$, so if we choose $\sigma_n\in\Sigma$ so that $a_n\in \mathcal{A}_{\sigma_n}$,
then by applying Lemma~\ref{seminorm} to the inclusion $\mathcal{B}\subset\mathcal{A}_{\sigma_n}$,
we get the inequalities
$\|E(a_n^2)-E(a_n)^2\|\leq 2\|a_n\|\cdot p(a_n)\leq \frac{4}n\left(\|x\|+\frac{1}n\right))$, which imply that
\begin{equation}
\lim_{n\to\infty}\|E(a_n^2)-E(a_n)^2\|=0.
\label{Ean-0}
\end{equation}
Since by continuity we have $\lim_{n\to \infty}E(a_n)=E(x)$ and
$\lim_{n\to \infty}E(a_n^2)=E(x^2)$, using \eqref{Ean-0} it follows that
$\|E(x^2)-E(x)^2\|=0$, so we get the identity:
\begin{equation}
E(x^2)=E(x)^2.
\label{Ex2=}
\end{equation}
By the properties of conditional expectations this yields:
$$E([x-E(x)]^2)=E(x^2-xE(x)-E(x)x+E(x)^2)=E(x^2)-E(x)^2=0.$$
In other words, if we consider the (self-adjoint) element $y=x-E(x)$, we have the equality
$E(y^*y)=0$, which by the faithfulness of $E$ implies $y=0$, which means that
$x=E(x)\in\mathcal{B}$.

{\sc (u)} Fix a $*$-representation $\pi:\mathcal{A}\to B(H)$, whose restriction
$\pi\big|_{\mathcal{B}}:\mathcal{B}\to B(H)$ is injective, and let us show that
$\pi$ itself is injective (thus isometric). Denote $\text{Ker}\,\pi$ simply by $\mathcal{J}$, so what we must show is the equality $\mathcal{J}=\{0\}$, given that
$\mathcal{J}\cap\mathcal{B}=\{0\}$. By construction, it suffices to prove that for each $\sigma\in\Sigma$ the restriction $\pi_\sigma=\pi\big|_{\mathcal{A}_\sigma}:\mathcal{A}_\sigma\to B(H)$ is injective (thus isometric), which is equivalent to the equality $\mathcal{J}\cap\mathcal{A}_\sigma=\{0\}$. Fix $\sigma$ and some $x\in\mathcal{J}\cap\mathcal{A}_\sigma$, and let us prove that
$x=0$. We can, of course, assume that $x$ is positive, so by the faithfulness of
$E$, it suffices to prove that $E_\sigma(x)=0$ (in $\mathcal{B}$).
By the density of $P_1(\mathcal{B}\!\uparrow\!\mathcal{A}_\sigma)$ in $P(\mathcal{B})$, it suffices
to prove that $\phi(E_\sigma(x))=0$,
$\forall\,\phi\in P_1(\mathcal{B}\!\uparrow\!\mathcal{A}_\sigma)$. Fix such
a $\phi$, and use Proposition~\ref{net}, to find a net
$(b^{}_\lambda)_{\lambda\in\Lambda}$ of positive elements in
$\mathcal{B}$, such that
$\|b^{}_\lambda\|=\phi(b^{}_\lambda)=1$, $\forall\,\lambda$, and such
that
\begin{equation}
\lim_{\lambda\in\Lambda}\|b^{}_\lambda
xb^{}_\lambda-\phi(E_\sigma(x))b^2_\lambda\|=0.
\label{uniq}\end{equation}
Remark now that, since $b^{}_\lambda xb^{}_\lambda$ belongs to
$\mathcal{J}(=\text{Ker}\,\pi)$, for every $\lambda$, by applying $\pi$ in
\eqref{uniq} we get
\begin{equation}
0=\lim_{\lambda\in\Lambda}\|\pi(\phi(E_\sigma(x))b^2_\lambda)\|=
\lim_{\lambda\in\Lambda}|\phi(E_\sigma(x))|\cdot\|\pi(b^{}_\lambda)^2\|.
\label{uniq2}
\end{equation}
But now, since we know that $\pi$ is injective (hence isometric) on $\mathcal{B}$,
we have
$$\|\pi(b^{}_\lambda)^2\|=\|\pi(b^{}_\lambda)\|^2=
\|b^{}_\lambda\|^2=1,\,\,\,\forall\,\lambda\in\Lambda,$$
and then \eqref{uniq2} forces $\phi(E(x))=0$, and we are done.
\end{proof}
\begin{corollary}
If $\mathcal{B}$ is a pseudo-diagonal in $\mathcal{A}$, then
$\mathcal{B}$ is an abelian core in $\mathcal{A}$.\hfill\qedsymbol
\end{corollary}

\begin{example}
Using the set-up from Example~\ref{coreex}, it follows that $M(E)$ is an abelian core in
$C^*(E)$. As pointed out in \cite{NaRe}, the fact that $M(E)$ has property {\sc (u)} gives an alternative proof of Szymanski's Uniqueness Theorem from \cite{Sz}, which in turn is
a generalization of the Uniqueness Theorem for Cuntz algebras.
\end{example}

The last application (compare with \cite[Corollary of Theorem~1]{AS}) is the following combined consequence of Theorems~\ref{crossprod} and~\ref{pdimpliesac}.

\begin{corollary}
If a (not necessarily countable) discrete group $G$ acts essentially freely on $X$, then
$C_0(X)$ is an abelian core in $C_0(X)\times_{\text{\rm red}}G$.

In particular, if the action of $G$ on $X$ is also minimal, i.e., the only $G$-invariant open subsets of $X$ are $\varnothing$ and $X$, then
$C_0(X)\times_{\text{\rm red}}G$ is simple.
\end{corollary}
\begin{proof}
As pointed out right before Lemma~\ref{evextends}, we do have a faithful conditional expectation $E_{\text{red}}$ of
$C_0(X)\times_{\text{red}}G$ onto $C_0(X)$.
Let $\Sigma$ be the collection of all countable subgroups of $G$, equipped with the order $H\succ K\Leftrightarrow H\supset K$, $H,K\in\Sigma$.
Since we work with discrete groups, for any subgroup $H\subset G$, we have an inclusion
$C_0(X)\times_{\text{red}}H\subset C_0(X)\times_{\text{red}}G$, so we can consider the net
$(C_0(X)\times_{\text{red}}H)_{H\in\Sigma}$ of C*-subalgebras in $C_0(X)\times_{\text{red}}G$, which clearly satisfies the equality (even without closure!)
$$C_0(X)\times_{\text{red}}G=\bigcup_{H\in\Sigma}C_0(X)\times_{\text{red}}H.$$
The first statement now follows from Theorem~\ref{pdimpliesac}, since by Theorem~\ref{crossprod}, each
inclusion $C_0(X)\subset C_0(X)\times_{\text{red}}H$ has the Almost Extension Property.

To prove the second statement, assume the action is minimal, and let us show that, if $\mathcal{J}$ is a closed two-sided ideal in $C_0(X)\times_{\text{red}}G$, then
either
\begin{itemize}
\item[$(a)$] $\mathcal{J}=C_0(X)\times_{\text{red}}G$, or
\item[$(b)$] $\mathcal{J}=\{0\}$.
\end{itemize}
Since $\mathcal{J}\cap C_0(X)$ is a $G$-invariant closed two-sided ideal in
$C_0(X)$, by minimality it follows that either
$(a')$ $\mathcal{J}\cap C_0(X)=C_0(X)$, or $(b')$ $\mathcal{J}\cap C_0(X)=\{0\}$.
In case $(a')$ it follows immediately that $\mathcal{J}=C_0(X)\times_{\text{red}}G$.
In case $(b')$, since $C_0(X)$ is an abelian core, by property {\sc (u)} it follows that
$\mathcal{J}=\{0\}$.
\end{proof}


\begin{thebibliography}{MM}
\bibitem[An]{And} J.~Anderson, {\em Extensions, restrictions, and representations of states on C*-algebras}, Trans.~A.M.S. 249 (1979).
\bibitem[AS]{AS} R.~Archbold, J.~Speielberg, {\em Topologically free actions and ideals in discrete C*-dynamical systems}, Proc.~Edinburgh Math.~Soc. 37 (1993).
\bibitem[KS]{KaSin} R.~Kadison, I.~Singer, {\em Extensions of pure states}, American J.~Math. 81 (1959).
\bibitem[Ku]{Kum} A.~Kumjian, {\em On C*-diagonals}, Canadian J.~Math 38 (1986).
\bibitem[Na]{Na} G.~Nagy, {\em Reduced crossed products via Hilbert modules}, Preprint 2010.
\bibitem[NR]{NaRe} G.~Nagy, S.~Reznikoff, {\em Abelian Core of Graph Algebras}, Preprint 2011.
\bibitem[Pa]{Pa} V.~Paulsen, {\em Three approaches to Kadison-Singer}, Preprint 2005
\bibitem[Ra]{Ra} I.~Raeburn, {\em Graph algebras},  CBMS Regional Conference Series in Mathematics {\bf 103} Published for the Conference Board of the Mathematical Sciences, Washington D.C.~by the AMS, Providence, RI (2005).
\bibitem[Re1]{Ren1} J.~Renault, {\em Cartan subalgebras in C*-algebras}, Irish Math.~Soc.~Bulletin 61 (2008).
\bibitem[Re2]{Ren2} J.~Renault, {\em Examples of masas in C*-algebras}, Preprint 2009.
\bibitem[Sz]{Sz} W.~Szymanski, {\em General Cuntz-Krieger uniqueness theorem}, International J.~Math. 13 (2002).
\bibitem[To]{To} J.~Tomyiama, {\em On the projection of norm one in W*-algebra}, Proc.~Japan Acad 33 (1957), 608-612.

\end{thebibliography}
\end{document}